\documentclass[dvipdfmx]{amsart}
\usepackage{amsmath, amssymb, amsthm, ascmac,tikz,bm,latexsym,here,stmaryrd,time,shuffle}

\usetikzlibrary{arrows}
\numberwithin{equation}{section}
 
\newtheorem{theorem}{Theorem}[section]
\newtheorem*{theorem*}{Theorem}
\newtheorem{lemma}[theorem]{Lemma}
\newtheorem{proposition}[theorem]{Proposition}
\newtheorem{corollary}[theorem]{Corollary}

\theoremstyle{definition}
\newtheorem{definition}[theorem]{Definition}
\newtheorem{remark}[theorem]{Remark}
\newtheorem{example}[theorem]{Example}

\newcommand{\N}{\mathbb{N}}
\newcommand{\Z}{\mathbb{Z}}
\newcommand{\Q}{\mathbb{Q}}
\newcommand{\R}{\mathbb{R}}
\newcommand{\C}{\mathbb{C}}
\newcommand{\h}{\mathfrak{h}}
\newcommand{\doubledecker}[2]{{}^{#1}_{#2}}

\title{Regularization of elliptic multiple zeta values}
\author{Taichi Katayama}
\address{Graduate School of Mathematics, Nagoya University,
Furo-cho, Chikusa-ku, Nagoya, 464-8602, Japan}
\email{taichi.katayama.e6@math.nagoya-u.ac.jp}
\date{April 21, 2025.}

\begin{document}
\maketitle

\begin{abstract}
In this paper, we show that regularized elliptic multiple zeta values are given by polynomials in elliptic multiple zeta values with admissible indices and special ones whose indices consist of $0$ and $1$. 
\end{abstract}

\tableofcontents

\setcounter{section}{-1}
\section{Introduction}
Multiple zeta values are real numbers defined by the series:
\[
\zeta(k_1,\ldots,k_r):=\sum_{0<m_1<\cdots<m_r}\frac{1}{m_1^{k_1}\ldots m_r^{k_r}}\;\;(k_1,\ldots,k_r\in\mathbb{Z}_{\geq 1}, k_r\geq2).
\]
They have been actively studied since the late 1990s and have found a connection in various fields such as knot theory \cite{LM}, motives \cite{DG}, and mathematical physics \cite{BK}. \par
Elliptic multiple zeta values were introduced in \cite{E2}, defined by iterated integrals over the family of punctured elliptic curves ${\left(\C/(\Z+\Z\tau)\setminus\{0\}\right)}_{\tau\in\h}$. These values are given by
\[
I^A(k_1,\ldots,k_r;\tau):=\int_{0<z_1<\cdots<z_r<1}f_\tau^{(k_1)}(z_1)dz_1\cdots f_\tau^{(k_r)}(z_r)dz_r
\]
where $\h$ denotes the upper half-plane $\{z\in\C\mid\Im(z)>0\}$, and $f_\tau^{(n)}(z)$ represents a specific type of function related to Eisenstein-Kronecker series: 
\begin{align*}
F_{\tau}\left(\doubledecker{\,\alpha\,}{\,z\,}\right)=\sum_{n\geq0}f_\tau^{(n)}(z)\alpha^{n-1}:=\frac{\theta_{\tau}(z+\alpha)\theta_{\tau}'(0)}{\theta_{\tau}(z)\theta_{\tau}(\alpha)}.
\end{align*}
The iterated integral converges when $k_1 \neq 1$ and $k_r \neq 1$. Such a tuple $(k_1,\ldots,k_r)$ is called an admissible index. In the case when $k_1=1$ or $k_r=1$, a technique called regularization, described later (cf. Definition 1.26), is used to define elliptic multiple zeta values $I^A(k_1,\ldots,k_r;\tau)$ for all indices $(k_1,\ldots,k_r)\in\N_0^r$ where $\N_0$ denotes the set of all non-negative integers.
\par
A geometric aspect of elliptic multiple zeta values arises from their interpretion in terms of the monodromy of the Knizhnik-Zamolodchikov-Bernard (KZB) connection. The KZB connection appeared as a higher-genus analogue of the KZ connection in the study of the Wess-Zumino-Witten model on elliptic curves in conformal field theory \cite{B1,B2}. Subsequently, it was reformulated for genus 1 cases in \cite{CEE}, where non-commutative formal power series called elliptic associator was introduced as a pair of connection matrices along two types of loops on $\left(\C/(\Z+\Z\tau)\right)\setminus\{0\}$. Just as in the genus 0 case, it was demonstrated in \cite{E2} that the elliptic associator serves as the generating function for elliptic multiple zeta values.
\par
In genus 0 case, the KZ associator was introduced as the connection matrix of the KZ connection along the real line segment from 0 to 1, and coefficients of the KZ associator are expressed in terms of so-called regularized iterated integrals. In \cite{LM}, it was demonstrated that the coefficients of the KZ associator can be written solely in terms of iterated integrals that converge in the usual sense. The main result of this paper is to establish a similar result for the elliptic associator.
To state the main result, we introduce the following notation. Let $\mathcal{E}^A\mathcal{Z}$ denote the $\mathbb{Q}$-linear space spanned by the regularized elliptic multiple zeta values:
\[
\mathcal{E}^A\mathcal{Z}:=\langle I^A(k_1,\ldots,k_r;\tau)\mid r,k_1,\ldots,k_r\in\N_0\rangle_\Q.
\]
Similarly, let $\mathcal{E}^A\mathcal{Z}_{\mathrm{adm}}$ denote the $\mathbb{Q}$-linear subspace spanned by the elliptic multiple zeta values with admissible indices:
\begin{align*}
\mathcal{E}^A\mathcal{Z}_{\mathrm{adm}}&:=\langle I^A(k_1,\ldots,k_r;\tau)\mid r,k_1,\ldots,k_r\in\N_0,k_1,k_r\neq1\rangle_\Q.
\end{align*}
One checks that $\mathcal{E}^A\mathcal{Z}$ and $\mathcal{E}^A\mathcal{Z}_{\mathrm{adm}}$ form $\mathbb{Q}$-algebras under the shuffle product of iterated integrals. 
\vspace{2mm}\\
\textbf{Theorem\;\ref{main}}\textbf{.}
\textit{Regularized elliptic multiple zeta values are expressed as polynomial combinations of elliptic multiple zeta values with admissible indices and some special elliptic multiple zeta values. More precisely, as $\mathbb{Q}$-algebras, we have the following equality:}
\[
\mathcal{E}^A\mathcal{Z}=\mathcal{E}^A\mathcal{Z}_{\mathrm{adm}}[I^A(k_1,\ldots,k_r;\tau)\mid r\in\N_0,k_1,\ldots,k_r\in\{0,1\}].
\]
\par
We established this theorem by extending the special case of Fay relations for elliptic multiple zeta values, as demonstrated in \cite{M1}, to a more general one which is introduced in Theorem \ref{emzv-fay}.
\section{Preliminaries}
In this section, we recall the elliptic multiple zeta values introduced in \cite{E2}. We recall the definition of elliptic multiple zeta values in §1.1 and review some of their properties in§1.2.
\subsection{Definition of elliptic multiple zeta values}
\begin{definition}[{cf. \cite{Z}}]
Fix $\tau\in{\h}$. The \textbf{Eisenstein-Kronecker series} is the 2 variable complex function $F_{\tau}\left(\doubledecker{\,\alpha\,}{\,z\,}\right)\;(\alpha,z\in\C)$ defined by
\[
F_{\tau}\left(\doubledecker{\,\alpha\,}{\,z\,}\right):=\frac{\theta_{\tau}(z+\alpha)\theta_{\tau}'(0)}{\theta_{\tau}(z)\theta_{\tau}(\alpha)}
\]
where $\theta_\tau$ is the odd Jacobi theta function given by
\[
\theta_{\tau}(z):=\sum_{n\in\Z}(-1)^nq^{\frac{1}{2}(n+\frac{1}{2})^2}w^{n+\frac{1}{2}} \hspace{7pt} (q:=e^{2\pi i\tau},w:=e^{2\pi iz}). 
\]
\end{definition}
The Eisenstein-Kronecker series possesses the following properties.
\begin{proposition}[{cf. \cite{Z}}]\label{prop:fay}
\begin{enumerate}
\renewcommand{\labelenumi}{(\roman{enumi})}
\item $F_\tau\left(\doubledecker{\,\alpha\,}{\,z\,}\right)=-F_\tau\left(\doubledecker{-\alpha}{-z}\right)$.
\item $F_{\tau}\left(\doubledecker{\,\alpha\,}{\,z\,}\right)$ is a holomorphic function on $\C^2\setminus\left((\Lambda_\tau\times\C)\cup(\C\times\Lambda_\tau)\right)$, where $\Lambda_\tau:=\Z+\Z\tau$. It has simple poles at $z=a+b\tau$ with residue $e^{-2\pi ib\alpha}$ and at $\alpha=a+b\tau$ with residue $e^{-2\pi ibz}$, for $a,b\in\mathbb{Z}$.
\item $F_{\tau}\left(\doubledecker{\;\;\alpha}{z+1}\right)=F_{\tau}\left(\doubledecker{\,\alpha\,}{\,z\,}\right)$ and $F_{\tau}\left(\doubledecker{\;\;\alpha}{z+\tau}\right)=e^{-2\pi i\alpha}F_{\tau}\left(\doubledecker{\,\alpha\,}{\,z\,}\right)$.
\item $F_\tau\left(\doubledecker{\,\alpha\,}{\,z\,}\right)$ satisfies the \textbf{Fay identity} for Eisenstein-Kronecker series:
\[
F_{\tau}\left(\doubledecker{\,\alpha_1\,}{\,z_1\,}\right)F_{\tau}\left(\doubledecker{\,\alpha_2\,}{\,z_2\,}\right)
    =F_{\tau}\left(\doubledecker{\alpha_1+\alpha_2}{\;\;\;z_1}\right)F_{\tau}\left(\doubledecker{\;\;\alpha_2}{z_2-z_1}\right)+F_{\tau}\left(\doubledecker{\alpha_1+\alpha_2}{\;\;\;z_2}\right)F_{\tau}\left(\doubledecker{\;\;\alpha_1}{z_1-z_2}\right).
\]
\end{enumerate}
\end{proposition}
By Proposition \ref{prop:fay} (ii), we put the Laurent expansion of $F_{\tau}\left(\doubledecker{\,\alpha\,}{\,z\,}\right)$ at $\alpha=0$ as follows:
\begin{align*}
F_{\tau}\left(\doubledecker{\,\alpha\,}{\,z\,}\right)=:\sum_{n\geq 0}f_{\tau}^{(n)}(z)\alpha^{n-1}.
\end{align*}
Each $f_n(z)$ is a meromorphic function, and for example, when $n = 0, 1$, they are given by
\[
f_\tau^{(0)}(z)=1, f_\tau^{(1)}(z)=\pi\cot(\pi z)+4\pi\sum_{k,l\geq1}\sin(2\pi kz)q^{kl}.
\]
Since the residue of $F_{\tau}\left(\doubledecker{\,\alpha\,}{\,z\,}\right)$ at $z=a\;(a\in\mathbb{Z})$ is $1$, which is independent of $\alpha$, $f_{\tau}^{(n)}(z)$ has a simple pole at $z=a\;(a\in\mathbb{Z})$ for only $n=1$.
\begin{lemma}
When $k_1 \neq 1$ and $k_r \neq 1$, the iterated integral \\$\int_{0<z_1<\cdots<z_r<1}f_\tau^{(k_1)}(z_1)dz_1\cdots f_\tau^{(k_r)}(z_r)dz_r$ converges.
\end{lemma}
\begin{proof}
It is enough to show that the iterated integral $\int_\gamma\omega_1\cdots\omega_r$ converges absolutely, where each $\omega_i$ is a meromorphic differential form such that it is analytic on $\gamma((0,1))$, has at most a simple pole at $\gamma(0)$ and $\gamma(1)$, and satisfies the conditions that $\omega_1$ has no pole at $\gamma(0)$ and $\omega_r$ has no pole at $\gamma(1)$. We will prove this by induction.
For the case $r=1$, the result is obvious.
Now, assume the result holds for such the iterated integrals of length less than $r$. By the path composition formula, we can replace the path with $\eta_1(t):=\gamma(t/2)$ and $\eta_2(t):=\gamma(\frac{1}{2}+t/2)$ to demonstrate convergence. Furthermore, convergence can be similarly proven for $\eta_2$ as for $\eta_1$, so here we demonstrate convergence for $\eta_1$ only. Let $\eta_1^\ast\omega_i=:g_i(z)dz$. Then, $g_1(z)$ is analytic on $0\leq z\leq1$, so there exists a positive real number $M$ such that $\sup_{0\leq z\leq1}|g_1(z)|\leq M$. Then we have
\begin{align*}
&\int_{0<z_1<\cdots<z_{r}<1}\left|g_1(z_1)dz_1\cdots g_{r}(z_{r})dz_{r}\right|\\
&\leq M\int_{0<z_2<\cdots<z_{r}<1}\left|z_2g_2(z_2)dz_2\cdots g_{r}(z_{r})dz_{r}\right|. 
\end{align*}
Since $g_2(z_2)$ has at most a simple pole at $z_2=0$, $z_2g_2(z_2)$ is analytic on $0\leq z_2\leq1$. By the induction hypothesis, we have $$\int_{0<z_2<\cdots<z_{k+1}<1}\left|z_2g_2(z_2)dz_2\cdots g_{r}(z_{r})dz_{r}\right|<\infty, $$which concludes the proof. 
\end{proof}
\begin{definition}[{\cite{E2}}]\label{def:emzv}
For $r\geq0$, $\bm{k}=(k_1,\ldots,k_r)\in\mathbb{N}_{0}^r$ with $k_1,k_r\neq1$ and $\tau\in\h$, the \textbf{elliptic multiple zeta value} is defined by the following iterated integral:
\begin{align}
I^A(\bm{k};\tau)=\int_{0<z_1<\cdots<z_r<1}f_\tau^{(k_1)}(z_1)dz_1\cdots f_\tau^{(k_r)}(z_r)dz_r\label{eq:emzvdef}
\end{align}
Moreover, we refer to $k_1+\cdots +k_r$ as the weight of $\bm{k}$, and $r$ as the length of $\bm{k}$. We treat $\emptyset$ as the index of length $0$ and weight $0$, and define $I^A(\emptyset;\tau):=1$.
\end{definition}
\begin{remark}
In \cite{E2}, another type of eMZV $I^B(\bm{k};\tau)$ is defined by the same iterated integral \eqref{eq:emzvdef} with the interval of integration $[0,1]$ replaced with $[0,\tau]$. However, it can be verified from \cite[(26)]{E2} that
\[
I^B(\bm{k};\tau)=\tau^{k_1+\cdots+k_r-n}I^A(\bm{k};-1/\tau). 
\]
\end{remark}
As mentioned in \cite[Remark 2.3.4]{M1}, the definition of $I^A(\bm{k};\tau)$ agrees with one of the regularizations of iterated integrals introduced in \cite{D4} for the tangential base points $(0,(-2\pi i)^{-1}\frac{\partial}{\partial z})$ and $(1,-(-2\pi i)^{-1}\frac{\partial}{\partial z})$. These base points correspond to $(1,-\frac{\partial}{\partial w})$ and $(1,\frac{\partial}{\partial w})$ on the Tate curve $\C^{\times}/q^{\Z}$ where $w=e^{2\pi iz}$ and $q=e^{2\pi i\tau}$.
\begin{definition}\label{def:asy}
Let $f:(0,1)\to\C$ be an analytic function. Assume $f$ has an asymptotic behavior of the form
\begin{align}
f(\epsilon)=\sum_{n=0}^{N}c_n(\mathrm{log}(-2\pi i\epsilon))^n+O(\epsilon^\delta).\label{eq:asy}
\end{align}
for sufficiently small $\epsilon>0$, where $\delta>0$, $N\in\N_0$, and $c_0,\ldots,c_N\in\C$ are constants. Then, the regularized value as $\epsilon\to0$ is defined by 
\[
\mathrm{Reg}_{\epsilon\to0}^{-2\pi i}f(\epsilon):=c_0\in\C.
\]
Here, the branch of the logarithm is taken such that $\mathrm{log}(-i)=-\frac{\pi i}{2}$.
\end{definition}
For any $k_1,\ldots,k_r\geq0$, the iterated integrals
\[
\int_{\epsilon<z_1<\cdots<z_r<1-\epsilon}f_\tau^{(k_1)}(z_1)dz_1\cdots f_\tau^{(k_r)}(z_r)dz_r
\]
have asymptotic behaviors of the form \eqref{eq:asy}. This follows from the fact that $f_\tau^{(1)}$ is a meromorphic function on a certain domain containing the interval $[0,1]$ with only simple poles at $0,1$ and $f_\tau^{(n)}\,(n\neq1)$ is  analytic on $[0,1]$. 
\begin{definition}[{\cite{E2}}]\label{def:reg}
For any $\bm{k}=(k_1,\ldots,k_r)\in\N_0^r$, we define $I^A(\bm{k};\tau)$ by
\begin{align*}
I^A(\bm{k};\tau)&:=\mathrm{Reg}_{\epsilon\to0}^{-2\pi i}\int_{\epsilon<z_1<\cdots<z_r<1-\epsilon}f_\tau^{(k_1)}(z_1)dz_1\cdots f_\tau^{(k_r)}(z_r)dz_r
\end{align*}
\end{definition}
From now on, we simply denote $I^A(\bm{k})$ for our abbreviation.
\begin{remark}
When $k_1,k_r\neq1$, since the limit $$\lim_{\epsilon\to0}\int_{\epsilon<z_1<\cdots<z_r<1-\epsilon}f_\tau^{(k_1)}(z_1)dz_1\cdots f_\tau^{(k_r)}(z_r)dz_r$$ converges, we have
\begin{align*}
&\mathrm{Reg}_{\epsilon\to0}^{-2\pi i}\int_{\epsilon<z_1<\cdots<z_r<1-\epsilon}f_\tau^{(k_1)}(z_1)dz_1\cdots f_\tau^{(k_r)}(z_r)dz_r\\&=\lim_{\epsilon\to0}\int_{\epsilon<z_1<\cdots<z_r<1-\epsilon}f_\tau^{(k_1)}(z_1)dz_1\cdots f_\tau^{(k_r)}(z_r)dz_r.
\end{align*}
Therefore, Definition \ref{def:reg} agrees with Definition \ref{def:emzv} in the case with $k_1,k_r\neq1$.
\end{remark}
\begin{example}
The elliptic multiple zeta value $I^A(k)\,(k\geq0)$ of length $1$ is given by
\begin{align*}
I^A(k)=
\begin{cases}
-2\zeta(k) &(k\text{ is even}),\\
0 &(k\text{ is odd})
\end{cases}
\end{align*}
and it is independent of $\tau$. Here, we put $\zeta(0)=-\frac{1}{2}$. The fact that $I^A(k)$ is zero when $k$ is odd follows from the reflection relation \eqref{eq:ref} mentioned later. When $k$ is even, it can be obtained from the $q$-expansion of the Eisenstein-Kronecker series. For details, refer to \cite[Proposition 3.18]{M1}.
\end{example}
\subsection{Some relations among elliptic multiple zeta values}
Let $\Q\langle\mathbf{e}\rangle$ be the unitary associative non-commutative $\Q$-algebra freely generated by $\mathbf{e}:=\{e_k\}_{k\in\N_0}$, and let $\mathbf{e}^*$ denote the free
monoid generated by $\mathbf{e}$. The shuffle product $\shuffle:\Q\langle\mathbf{e}\rangle^{\otimes2}\to\Q\langle\mathbf{e}\rangle$ is recursively defined as follows:
\begin{align*}
w\shuffle1 &= 1\shuffle w = w \quad (\text{for } w\in\mathbf{e}^*),\\
e_i v' \shuffle e_j w' &= e_i(v' \shuffle e_j w') + e_j(e_i v' \shuffle w') \quad (\text{for } v', w'\in\mathbf{e}^*, i, j\geq0).
\end{align*}
$\Q\langle\mathbf{e}\rangle$ forms a $\Q$-algebra with respect to $\shuffle$. We denote this algebra in order to distinguish from the usual product, as $(\Q\langle\mathbf{e}\rangle,\shuffle)$. The shuffle product for elements of $\mathbf{e}^*$ is expressed as follows:
\[
e_{i_1}\cdots e_{i_n}\shuffle e_{i_{n+1}}\cdots e_{i_{n+m}}=\sum_{\sigma\in\mathcal{S}_{n,m}}e_{i_{\sigma(1)}}\cdots e_{i_{\sigma(n+m)}}.
\]
where $\mathcal{S}_{n,m}$ denotes the set of all permutations of $\mathfrak{S}_{n+m}$ such that
\[
\sigma^{-1}(1)<\cdots<\sigma^{-1}(n),\sigma^{-1}(n+1)<\cdots<\sigma^{-1}(n+m).
\]
Furthermore, we define the coproduct $\Delta:\Q\langle\mathbf{e}\rangle\to\Q\langle\mathbf{e}\rangle^{\otimes2}$ as follows:
\[
\Delta(e_{i_1}\cdots e_{i_n})=\sum_{j=0}^ne_{i_1}\cdots e_{i_j}\otimes e_{i_{j+1}}\cdots e_{i_n}.
\]
 Also, we define the antipode map $S:\Q\langle\mathbf{e}\rangle\to\Q\langle\mathbf{e}\rangle$ as $S(e_i)=-e_i$. It is known that $(\Q\langle\mathbf{e}\rangle,\shuffle)$ forms a Hopf algebra. Let $L^A:(\Q\langle\mathbf{e}\rangle,\shuffle)\to\mathcal{E}^A\mathcal{Z}$ be the $\Q$-linear map defined by
\[
L^A(e_{i_1}\cdots e_{i_n}):=I^A(i_1,\ldots,i_n), 
\]
where
\[
\mathcal{E}^A\mathcal{Z}:=\langle I^A(k_1,\ldots,k_r)\mid r,k_1,\ldots,k_r\in\N_0\rangle_\Q
\]
denotes the $\Q$-linear space (actually $\Q$-algebra) spanned by the elliptic multiple zeta values.
\begin{proposition}[{\cite{E2}}]\label{prop:sh}
The map $L^A:(\Q\langle\mathbf{e}\rangle,\shuffle)\to\mathcal{E}^A\mathcal{Z}$ is a $\Q$-algebra homomorphism.
\end{proposition}
\begin{proof}
For $e_{i_1},\ldots,e_{i_{n+m}}\in\mathbf{e}$, we need to show that $L^A(e_{i_1}\cdots e_{i_n})L^A(e_{i_{n+1}}\cdots e_{i_{n+m}})=L^A(e_{i_1}\cdots e_{i_n}\shuffle e_{j_1}\cdots e_{j_m})$, or equivalently,
\begin{align}
I^A({i_1},\ldots,{i_n})I^A({i_{n+1}},\ldots,{i_{n+m}})=\sum_{\sigma\in\mathcal{S}_{n,m}}I^A({i_{\sigma(1)}},\ldots,{i_{\sigma(n+m)}}).\label{eq:sh}
\end{align}
This follows directly from the shuffle product of iterated integrals.
\end{proof}
\begin{remark}
Equation \eqref{eq:sh} is called the \textbf{shuffle relation} for elliptic multiple zeta values.
\end{remark}
\begin{proposition}[{\cite{M1}}]\label{prop:ref}
For $w=e_{i_1}\cdots e_{i_{n}}\in\mathbf{e}^*$, we have $L^A(S(w))=(-1)^{i_1+\cdots i_n+n}L^A(w)$.
\end{proposition}
\begin{proof}
We need to show that 
\begin{align}
I^A(i_n,\ldots,i_1)=(-1)^{i_1+\cdots i_n}I^A(i_1,\ldots,i_n)\label{eq:ref}
\end{align}
which follows from the following calculation:
\begin{align*}
&I^A(i_n,\ldots,i_1)\\
&=\int_{0<z_n<\cdots<z_1<1}f_\tau^{(i_n)}(z_n)dz_n\cdots f_\tau^{(i_1)}(z_1)dz_1\\
&=(-1)^n\int_{0<z_1<\cdots<z_n<1}f_\tau^{(i_1)}(1-z_1)d(1-z_1)\cdots f_\tau^{(i_n)}(1-z_n)d(1-z_n)\\
&=\int_{0<z_1<\cdots<z_n<1}f_\tau^{(i_1)}(1-z_1)dz_1\cdots f_\tau^{(i_n)}(1-z_n)dz_n\\
\intertext{Here, we have $f_\tau^{(k)}(1-z)=(-1)^kf_\tau^{(k)}(z)$  By using Proposition \ref{prop:fay} (i) and (iii), then}
&=(-1)^{i_1+\cdots+i_n}\int_{0<z_1<\cdots<z_n<1}f_\tau^{(i_1)}(z_1)dz_1\cdots f_\tau^{(i_n)}(z_n)dz_n\\
&=(-1)^{i_1+\cdots+i_n}I^A(i_1,\ldots,i_n).
\end{align*}
\end{proof}
\begin{remark}
Equation \eqref{eq:ref} is called the \textbf{reflection relation} for elliptic multiple zeta values. From the reflection relation, it can be confirmed that $I^A(k)=0$ for $k\in\N_0$ odd.
\end{remark}
By keeping the identiy $L^A\circ\varphi^{-1}(\bm{x})=I^A(\bm{x})$ where $\varphi$ is the $\Q$-linear isomorphism$$\varphi:\Q\langle\mathbf{e}\rangle\to\oplus_{r\geq0}\langle\N_0^r\rangle_\Q;\,e_{i_1}\cdots e_{i_n}\mapsto(i_1,\ldots,i_n),$$we extend $I^A(\bm{x})$ for $\bm{x}\in\oplus_{r\geq0}\langle\N_0^r\rangle_\Q$ . Finally, it is known that elliptic multiple zeta values satisfy not only shuffle relations and reflection relations but also additional relations known as Fay relations, which are derived from the Fay identity of Eisenstein-Kronecker series (Proposition \ref{prop:fay} (iv)). Below, we introduce an explicit formula demonstrated in \cite{M1} for the case of length 2.
\begin{proposition}[{\cite{M1}}]\label{prop:mat}
For any $r,s\in\N_0$ with $(r,s)\neq(1,1)$, we have:
\begin{align*}
I^A(r,s)&=-(-1)^sI^A(0,r+s)+\sum_{n=0}^{s}(-1)^{s-n}\binom{r-1+n}{r-1}I^A(r+n,s-n)\\&+\sum_{n=0}^{r}(-1)^{s+n}\binom{s-1+n}{s-1}I^A(s+n,r-n).
\end{align*}
\end{proposition}
We will provide a proof for more general case in §2.1 (cf. Theorem \ref{emzv-fay}). It is conjectured in \cite{BMS} that the relations among elliptic multiple zeta values are exhausted by the shuffle relations, reflection relations, and Fay relations. The case for length 2 has been proven in \cite{M1}. Determining the $\Q$-linear relations for elliptic multiple zeta values of length 3 or higher remains an open problem. 
\section{Elliptic multiple polylogarithms}
In this section, we recall elliptic multiple polylogarithms and investigate their properties. In §2.1, we recall the relations between elliptic multiple polylogarithms and elliptic multiple zeta values. Subsequently, in §2.2, we prove the relations obtained from the differential equations of elliptic multiple polylogarithms.
\subsection{Relations between elliptic multiple polylogarithms and elliptic multiple zeta values}
\begin{definition}
For $\tau\in\h$ and $0<z<1$, with $a_i\in\{0,z\}$, we define $A_\tau\left(\doubledecker{u_1,\ldots,u_r}{a_1,\ldots,a_r};z\right)$ by the following iterated integral:
\begin{align*}
A_\tau\left(\doubledecker{u_1,\ldots,u_r}{a_1,\ldots,a_r};z\right)&:=\mathrm{Reg}_{\epsilon\to0}^{-2\pi i}\int_{\epsilon<z_1<\cdots<z_r<z-\epsilon}F_\tau(\doubledecker{\;\;\;u_1}{z_1-a_1})dz_1\cdots F_\tau(\doubledecker{\;\;\;u_r}{z_r-a_r})dz_r\\&\in(u_1\cdots u_r)^{-1}\C[[u_1,\ldots,u_r]].
\end{align*}
\end{definition}
We denote the coefficient of $u_1^{k_1-1}\cdots u_r^{k_r-1}$ in $A_\tau\left(\doubledecker{u_1,\ldots,u_r}{a_1,\ldots,a_r};z\right)$ as $\Gamma_\tau\left(\doubledecker{k_1,\ldots,k_r}{a_1,\ldots,a_r};z\right)$. In other words,
\[
A_\tau\left(\doubledecker{u_1,\ldots,u_r}{a_1,\ldots,a_r};z\right)=\sum_{k_1,\ldots,k_r\geq0}\Gamma_\tau\left(\doubledecker{k_1,\ldots,k_r}{a_1,\ldots,a_r};z\right)u_1^{k_1-1}\cdots u_r^{k_r-1}.
\]
The iterated integral $\Gamma_\tau\left(\doubledecker{k_1,\ldots,k_r}{a_1,\ldots,a_r};z\right)$ is referred to as an \textit{elliptic multiple polylogarithm} in \cite{KR}.
\begin{definition}
For $m,n\geq1$, let $\Omega(m,n)$ be the set of functions $\varphi:(0,1)^{m}\times[0,1)^{n}\to\C$ given by the convergent series
\begin{align*}
\varphi(x_1,\cdots,x_m;y_1,\cdots,y_n)&=\sum_{i_1,\ldots,i_m=0}^{M}\left(\sum_{j_1,\ldots,j_n=0}^{\infty}c_{i_1,\ldots,i_m,j_1,\ldots,j_n}y_1^{j_1}\cdots y_n^{j_n}\right)\\&\cdot(\log(-2\pi ix_1))^{i_1}\cdots (\log(-2\pi ix_m))^{i_m}
\end{align*}
for some $M\in\N_0$ and coefficients $\{c_{i_1,\ldots,i_m,j_1,\ldots,j_n}\}\subset\C$. Here, the branch of the logarithm is taken such that $\mathrm{log}(-i)=-\frac{\pi i}{2}$.
\end{definition}
\begin{lemma}\label{lem:iia}Let $g_i(t)\;(1\leq i\leq n)$ be meromorphic functions in a neighborhood of $t=0$, which have at most a simple pole at $t=0$ and are analytic on $(0,1)$. Then there exists $\varphi\in\Omega(2,2)$ such that, for any real numbers $0<a\leq b<1$, 
$$
\int_{a<t_1<\cdots<t_n<b}g_1(t_1)dt_1\cdots g_n(t_n)dt_n=\varphi(a,b;a,b). 
$$
\end{lemma}
\begin{proof}
We prove the statement by induction on $n$. For $n=1$, since $g_1(t)-\frac{c_{0}}{t}$ (where $c_{0}$ is the residue of $g_1(t)$ at $t=0$) is given by an absolutely convergent power series on $a\leq t\leq b$,
$$g_1(t)-\frac{c_{0}}{t}=\sum_{k\geq0}c_{k+1}t^{k}.$$
Thus, we have
\begin{align*}
&\int_{a<t_1<b}g_1(t_1)dt_1\\
&=\int_{a<t<b}\sum_{k\geq0}c_kt^{k-1}\\
&=c_0(\log(b)-\log(a))+\sum_{k\geq1}\frac{c_{k}}{k}(b^{k}-a^{k})\\
&=c_0(\log(-2\pi ib)-\log(-2\pi ia))+\sum_{k\geq1}\frac{c_{k}}{k}(b^{k}-a^{k}),
\end{align*}
which satisfies the required condition. 

Suppose that the statement holds for $n-1$. By induction hypothesis, there exists $\varphi\in\Omega(2,2)$ such that 
$$
\int_{a<t_2<\cdots<t_n<b}g_2(t_2)dt_2\cdots g_n(t_n)dt_n=\varphi(a,b;a,b)
$$
for any real numbers $0<a\leq b<1$. Thus, we have
\begin{align*}
&\int_{a<t_1<\cdots<t_n<b}g_1(t_1)dt_1\cdots g_n(t_n)dt_n\\
&=\int_{a<t_1<b}g_1(t_1)\varphi(t_1,b;t_1,b)dt_1. 
\end{align*}
Since $g_1(t)-\frac{c_{0}}{t}$ is given by an absolutely convergent power series on $a\leq t\leq b$,  the statement follows. 
\end{proof}
\begin{lemma}\label{lem:fps}
{\rm(i)} For any $\bm{k}=(k_1,\ldots,k_r)\in\N_0^r$, there exists $\varphi\in\Omega(2,2)$ such that
\[
\int_{\epsilon<z_1<\cdots<z_r<z-\eta}f_\tau^{(k_1)}(z_1)dz_1\cdots f_\tau^{(k_r)}(z_r)dz_r=\varphi(\epsilon,1-z+\eta;\epsilon,1-z+\eta)
\]
for any real numbers $0<z\leq 1$ and $\epsilon,\eta>0$ satisfying $0<\epsilon\leq z-\eta<1$. \\
{\rm(ii)} For any $k\in\N_0$, there exists $\varphi\in\Omega(2,2)$ such that
\[
\int_{\epsilon<t<z-\eta}f_\tau^{(k)}(t-z)dt=\varphi(1-z+\epsilon,\eta;1-z+\epsilon,\eta)
\]
for any real numbers $0<z\leq 1$ and $\epsilon,\eta>0$ satisfying $0<\epsilon\leq z-\eta<1$. \\
{\rm(iii)} For any $\bm{k}=(k_1,\ldots,k_r)\in\N_0^r$ with $r>1, k_r\neq1$, there exists $\varphi\in\Omega(2,3)$ such that
\begin{align*}
&\int_{\epsilon<z_1<\cdots<z_r<z-\eta}f_\tau^{(k_1)}(z_1)dz_1\cdots f_\tau^{(k_{r-1})}(z_{r-1})dz_{r-1}f_\tau^{(k_r)}(z_r-z)dz_r\\&=\varphi(\epsilon,1-z+\eta;\epsilon,1-z+\eta,1-z)
\end{align*}
for any real numbers $0<z\leq 1$ and $\epsilon,\eta>0$ satisfying $0<\epsilon\leq z-\eta<1$. 
\end{lemma}
\begin{proof}
(i) Using the path composition formula, we have
\begin{align*}
&\int_{\epsilon<z_1<\cdots<z_r<z-\eta}f_\tau^{(k_1)}(z_1)dz_1\cdots f_\tau^{(k_r)}(z_r)dz_r\\&=\sum_{i=0}^{r}\int_{\epsilon<z_1<\cdots<z_i<\frac{1}{2}}f_\tau^{(k_1)}(z_1)dz_1\cdots f_\tau^{(k_i)}(z_i)dz_i\\&\cdot\int_{\frac{1}{2}<z_{i+1}<\cdots<z_r<z-\eta}f_\tau^{(k_{i+1})}(z_{i+1})dz_{i+1}\cdots f_\tau^{(k_r)}(z_r)dz_r\\
&=\sum_{i=0}^{r}\int_{\epsilon<z_1<\cdots<z_i<\frac{1}{2}}f_\tau^{(k_1)}(z_1)dz_1\cdots f_\tau^{(k_i)}(z_i)dz_i\\&\cdot\int_{1-z+\eta<w_{r}<\cdots<w_{i+1}<\frac{1}{2}}f_\tau^{(k_{r})}(1-w_{r})dw_r\cdots f_\tau^{(k_{i+1})}(1-w_{i+1})dw_{i+1}. 
\end{align*}
For each $i\;(0\leq i\leq r)$, by Lemma \ref{lem:iia}, there exist $\varphi_i\in\Omega(2,2)$ and $\psi_i\in\Omega(2,2)$ (where we set $\varphi_0=1$ and $\psi_r=1$) such that 
$$
\int_{\epsilon<z_1<\cdots<z_i<\frac{1}{2}}f_\tau^{(k_1)}(z_1)dz_1\cdots f_\tau^{(k_i)}(z_i)dz_i=\varphi_i(\epsilon,\frac{1}{2};\epsilon,\frac{1}{2})
$$
and
$$
\int_{1-z+\eta<w_{r}<\cdots<w_{i+1}<\frac{1}{2}}f_\tau^{(k_{r})}(1-w_{r})dw_r\cdots f_\tau^{(k_{i+1})}(1-w_{i+1})dw_{i+1}=\psi_i(1-z+\eta,\frac{1}{2};1-z+\eta,\frac{1}{2})
$$
for any real numbers $0<z\leq 1$ and $\epsilon,\eta>0$ satisfying $0<\epsilon\leq z-\eta<1$. Thus, we have
\begin{align*}
&\int_{\epsilon<z_1<\cdots<z_r<z-\eta}f_\tau^{(k_1)}(z_1)dz_1\cdots f_\tau^{(k_r)}(z_r)dz_r=\sum_{i=0}^{r}\varphi_i(\epsilon,\frac{1}{2};\epsilon,\frac{1}{2})\psi_i(1-z+\eta,\frac{1}{2};1-z+\eta,\frac{1}{2}), 
\end{align*}
which leads to the desired result. 
\\(ii)  When $k\neq1$, since $f_\tau^{(k_1)}(x)$ is an analytic function at $x=0$ and $x=-1$, we have
\begin{align*}
\int_{\epsilon<t<z-\eta}f_\tau^{(k)}(t-z)dt=F(-\eta)-F(\epsilon-z)
\end{align*}
where $F(x)$ is an antiderivative of $f_\tau^{(k)}(x)$, which is also analytic at $x=0$ and $x=-1$. Thus, $F(-\eta)$ is given by a convergent power series in $\eta$ and $F(\epsilon-z)$ is given by a convergent power series in $1-z+\epsilon$. Therefore we obtain the result. When $k=1$, since $f_\tau^{(1)}(x)-\frac{1}{x}-\frac{1}{x+1}$ is analytic at $x=0$ and $x=-1$, we have
\begin{align*}
&\int_{\epsilon<t<z-\eta}f_\tau^{(1)}(t-z)dt\\
&=\tilde{F}(-\eta)-\tilde{F}(\epsilon-z)-\log(\eta)+\log(z-\epsilon)-\log(1-\eta)+\log(1-z+\epsilon)\\
&=\tilde{F}(-\eta)-\tilde{F}(\epsilon-z)-\log(-2\pi i\eta)+\log(1-(1-z+\epsilon))-\log(1-\eta)+\log(-2\pi i(1-z+\epsilon))
\end{align*}
where $\tilde{F}(x)$ is an antiderivative of $f_\tau^{(1)}(x)-\frac{1}{x}-\frac{1}{x+1}$, which is also analytic at $x=0$ and $x=-1$. Thus, $\tilde{F}(-\eta)$ is given by a convergent power series in $\eta$ and $\tilde{F}(\epsilon-z)$ is given by a convergent power series in $1-z+\epsilon$. This completes the proof for this case. \\
(iii) We proceed by induction on $r\geq1$. For $\bm{k}=(k_1,\ldots,k_r)\in\N_0^r$ with $r>1, k_r\neq1$, we have
\begin{align*}
&\int_{\epsilon<z_1<\cdots<z_r<z-\eta}f_\tau^{(k_1)}(z_1)dz_1\cdots f_\tau^{(k_{r-1})}(z_{r-1})dz_{r-1}f_\tau^{(k_r)}(z_r-z)dz_r\\
&=\int_{\epsilon<z_r<z-\eta}\left(\int_{\epsilon<z_1<\cdots z_{r-1}<z_r}f_\tau^{(k_1)}(z_1)dz_1\cdots f_\tau^{(k_{r-1})}(z_{r-1})dz_{r-1}\right)f_\tau^{(k_r)}(z_r-z)dz_r. 
\end{align*}
By Lemma \ref{lem:iia}, there exists $\varphi_1\in\Omega(2,2)$ such that
\begin{align*}
\int_{\epsilon<z_1<\cdots<z_{r-1}<z_r}f_\tau^{(k_1)}(z_1)dz_1\cdots f_\tau^{(k_{r-1})}(z_{r-1})dz_{r-1}=\varphi_1(\epsilon,z_r;\epsilon,z_r)
\end{align*}
for any real numbers $0<\epsilon\leq z_r<1$. 
Since $k_r\neq1$, $f_\tau^{(k_r)}$ is analytic at any points in $\R$. Thus, $f_\tau^{(k_r)}(z_r-z)$ is given by 
$$
f_\tau^{(k_r)}(z_r-z)=\sum_{n\geq0}c_n(z)z_r^n
$$
where $c_n(z):=\frac{1}{n!}\frac{d^n}{dx^n}f_\tau^{(k_r)}(x)\mid_{x=-z}$. Since $f_\tau^{(k_r)}$ is periodic with period 1, each $c_n(z)$ is given by a convergent series in $1-z$. Hence, there exists $\psi_1\in\Omega(2,3)$ such that
\begin{align*}
\varphi_1(\epsilon,z_r;\epsilon,z_r)f_\tau^{(k_r)}(z_r-z)=\psi_1(\epsilon,z_r;\epsilon,z_r,1-z)
\end{align*}
for any real numbers $0<\epsilon\leq z_r<z\leq1$. On the other hand, in the same way as (i), there exists $\varphi_2\in\Omega(2,2)$ such that
\begin{align*}
\int_{\epsilon<z_1<\cdots<z_{r-1}<z_r}f_\tau^{(k_1)}(z_1)dz_1\cdots f_\tau^{(k_{r-1})}(z_{r-1})dz_{r-1}=\varphi_2(\epsilon,1-z_r;\epsilon,1-z_r)
\end{align*}
for any real numbers $0<\epsilon\leq z_r<1$. Since $f_\tau^{(k_r)}$ is analytic at any points in $\R$, $f_\tau^{(k_r)}(z_r-z)$ is given by 
$$
f_\tau^{(k_r)}(z_r-z)=\sum_{n\geq0}d_n(z)(1-z_r)^n
$$
where $d_n(z):=\frac{(-1)^n}{n!}\frac{d^n}{dx^n}f_\tau^{(k_r)}(x)\mid_{x=1-z}$. Since each $d_n(z)$ is given by a convergent series in $1-z$, there exists $\psi_2\in\Omega(2,3)$ such that
\begin{align*}
\varphi_2(\epsilon,1-z_r;\epsilon,1-z_r)f_\tau^{(k_r)}(z_r-z)=\psi_2(\epsilon,1-z_r;\epsilon,1-z_r,1-z)
\end{align*}
for any real numbers $0<\epsilon\leq z_r<z\leq1$. 
Therefore, there exist $\Psi_1\in\Omega(2,3)$ and $\Psi_2\in\Omega(2,3)$ such that
$$
G(z_r)=\Psi_1(\epsilon,z_r;\epsilon,z_r,1-z)=\Psi_2(\epsilon,1-z_r;\epsilon,1-z_r,1-z)
$$
for any real numbers $0<\epsilon\leq z_r<z\leq1$, where $G(z_r)$ is an antiderivative of $\left(\int_{\epsilon<z_1<\cdots<z_{r-1}<z_r}f_\tau^{(k_1)}(z_1)dz_1\cdots f_\tau^{(k_{r-1})}(z_{r-1})dz_{r-1}\right)f_\tau^{(k_r)}(z_r-z)$. Hence, we have
\begin{align*}
&\int_{\epsilon<z_r<z-\eta}\left(\int_{\epsilon<z_1<\cdots<z_{r-1}<z_r}f_\tau^{(k_1)}(z_1)dz_1\cdots f_\tau^{(k_{r-1})}(z_{r-1})dz_{r-1}\right)f_\tau^{(k_r)}(z_r-z)dz_r\\
&=G(z-\eta)-G(\epsilon)\\
&=\Psi_2(\epsilon,1-z+\eta;\epsilon,1-z+\eta,1-z)-\Psi_1(\epsilon,\epsilon;\epsilon,\epsilon,1-z), 
\end{align*}
which completes the proof. 
\end{proof}
Elliptic multiple polylogarithms are related to elliptic multiple zeta values as follows. 
\begin{lemma}\label{lem:zreg}
{\rm(i)} For any $\bm{k}=(k_1,\ldots,k_r)\in\N_0^r$, we have
\begin{align*}
\mathrm{Reg}_{1-z\to0}^{-2\pi i}\Gamma_\tau \left(\doubledecker{k_1,\ldots,k_{r}}{0\;,\ldots,\;0\;};z\right)=I^A(k_1,\ldots,k_r).
\end{align*}
{\rm(ii)}  For any $\bm{k}=(k_1,\ldots,k_r)\in\N_0^r$ where $r=1$ or $r>1, k_r\neq1$, we have
\begin{align*}
\mathrm{Reg}_{1-z\to0}^{-2\pi i}\Gamma_\tau\left(\doubledecker{k_1,\ldots,k_{r-1},k_r}{\;0\;,\ldots,\;\;0\;\;\;,\;z};z\right)=I^A(k_1,\ldots,k_r).
\end{align*}
\end{lemma}
\begin{proof}
(i) By Lemma \ref{lem:fps} (i), the integral $$\int_{\epsilon<z_1<\cdots<z_r<z-\epsilon}f_\tau^{(k_1)}(z_1)dz_1\cdots f_\tau^{(k_r)}(z_r)dz_r$$ is given by the convergent series
\begin{align*}
&\int_{\epsilon<z_1<\cdots<z_r<z-\epsilon}f_\tau^{(k_1)}(z_1)dz_1\cdots f_\tau^{(k_r)}(z_r)dz_r\\
&=\sum_{i_1,i_2=0}^{M}\left(\sum_{j_1,j_2=0}^{\infty}c_{i_1,i_2,j_1,j_2}\epsilon^{j_1}(1-z+\epsilon)^{j_2}\right)(\log(-2\pi i\epsilon))^{i_1}(\log(-2\pi i(1-z+\epsilon)))^{i_2}
\end{align*}
for any real numbers $0<z\leq 1$ and $\epsilon>0$ satisfying $0<\epsilon\leq z-\epsilon<1$. 
Thus, we have
\begin{align*}
&\mathrm{Reg}_{1-z\to0}^{-2\pi i}\Gamma_\tau \left(\doubledecker{k_1,\ldots,k_{r}}{0\;,\ldots,\;0\;};z\right)\\
&=\mathrm{Reg}_{1-z\to0}^{-2\pi i}\mathrm{Reg}_{\epsilon\to0}^{-2\pi i}\int_{\epsilon<z_1<\cdots<z_r<z-\epsilon}f_\tau^{(k_1)}(z_1)dz_1\cdots f_\tau^{(k_r)}(z_r)dz_r\\
&=\mathrm{Reg}_{1-z\to0}^{-2\pi i}\sum_{i_2=0}^{M}\left(\sum_{j_2=0}^{\infty}c_{0,i_2,0,j_2}(1-z)^{j_2}\right)(\log(-2\pi i(1-z)))^{i_2}\\
&=c_{0,0,0,0}\\
&=\mathrm{Reg}_{\epsilon\to0}^{-2\pi i}\sum_{i_1,i_2=0}^{M}\left(\sum_{j_1,j_2=0}^{\infty}c_{i_1,i_2,j_1,j_2}\epsilon^{j_1+j_2}\right)(\log(-2\pi i\epsilon))^{i_1+i_2}\\
&=\mathrm{Reg}_{\epsilon\to0}^{-2\pi i}\int_{\epsilon<z_1<\cdots<z_r<1-\epsilon}f_\tau^{(k_1)}(z_1)dz_1\cdots f_\tau^{(k_r)}(z_r)dz_r\\
&=I^A(k_1,\ldots,k_r).
\end{align*}
(ii) By Lemma \ref{lem:fps} (ii) and (iii), the integral $$\int_{\epsilon<z_1<\cdots<z_r<z-\epsilon}f_\tau^{(k_1)}(z_1)dz_1\cdots f_\tau^{(k_{r-1})}(z_{r-1})dz_{r-1}f_\tau^{(k_r)}(z_r-z)dz_r$$ is given by the convergent series
\begin{align*}
&\int_{\epsilon<z_1<\cdots<z_r<z-\epsilon}f_\tau^{(k_1)}(z_1)dz_1\cdots f_\tau^{(k_{r-1})}(z_{r-1})dz_{r-1}f_\tau^{(k_r)}(z_r-z)dz_r\\
&=\sum_{i_1,i_2=0}^{N}\left(\sum_{j_1,j_2,j_3=0}^{\infty}d_{i_1,i_2,j_1,j_2,j_3}\epsilon^{j_1}(1-z+\epsilon)^{j_2}(1-z)^{j_3}\right)(\log(-2\pi i\epsilon))^{i_1}(\log(-2\pi i(1-z+\epsilon)))^{i_2}
\end{align*}
for any real numbers $0<z\leq 1$ and $\epsilon>0$ satisfying $0<\epsilon\leq z-\epsilon<1$. When $r=1$, note that the summation over $j_3$ is taken only for $j_3=0$. 
Then we have
\begin{align*}
&\mathrm{Reg}_{1-z\to0}^{-2\pi i}\Gamma_\tau\left(\doubledecker{k_1,\ldots,k_{r-1},k_r}{\;0\;,\ldots,\;\;0\;\;\;,\;z};z\right)\\
&=\mathrm{Reg}_{1-z\to0}^{-2\pi i}\mathrm{Reg}_{\epsilon\to0}^{-2\pi i}\int_{\epsilon<z_1<\cdots<z_r<z-\epsilon}f_\tau^{(k_1)}(z_1)dz_1\cdots f_\tau^{(k_{r-1})}(z_{r-1})dz_{r-1}f_\tau^{(k_r)}(z_r-z)dz_r\\
&=\mathrm{Reg}_{1-z\to0}^{-2\pi i}\sum_{i_2=0}^{N}\left(\sum_{j_2,j_3=0}^{\infty}d_{0,i_2,0,j_2,j_3}(1-z)^{j_2+j_3}\right)(\log(-2\pi i(1-z)))^{i_2}\\
&=d_{0,0,0,0,0}\\
&=\mathrm{Reg}_{\epsilon\to0}^{-2\pi i}\sum_{i_1,i_2=0}^{N}\left(\sum_{j_1,j_2=0}^{\infty}d_{i_1,i_2,j_1,j_2,0}\epsilon^{j_1+j_2}\right)(\log(-2\pi i\epsilon))^{i_1+i_2}\\
&=\mathrm{Reg}_{\epsilon\to0}^{-2\pi i}\int_{\epsilon<z_1<\cdots<z_r<1-\epsilon}f_\tau^{(k_1)}(z_1)dz_1\cdots f_\tau^{(k_{r-1})}(z_{r-1})dz_{r-1}f_\tau^{(k_r)}(z_r-1)dz_r\\
&=\mathrm{Reg}_{\epsilon\to0}^{-2\pi i}\int_{\epsilon<z_1<\cdots<z_r<1-\epsilon}f_\tau^{(k_1)}(z_1)dz_1\cdots f_\tau^{(k_{r-1})}(z_{r-1})dz_{r-1}f_\tau^{(k_r)}(z_r)dz_r\\
&=I^A(k_1,\ldots,k_r). 
\end{align*}
\end{proof}
\subsection{Differential relations among elliptic multiple polylogarithms}
Elliptic multiple polylogarithms satisfy the following differential equation. The proof of this differential equation involves Fay identity for Eisenstein-Kronecker series.
\begin{proposition}[{cf. \cite{BMMS}}]\label{prop:diff}
The following differential equation holds:
\begin{align*}
{d}A_\tau\left(\doubledecker{u_1,\ldots,u_r}{a_1,\ldots,a_r};z\right)
=\sum_{i=1}^r&\Big(A_\tau(\doubledecker{u_1,\ldots,u_{i-1},u_i+u_{i+1},\ldots,u_r}{a_1,\ldots,a_{i-1},\;\;\;\;a_{i+1}\;\;\;,\ldots,a_r};a_{r+1})F_\tau(\doubledecker{\;\;\;u_i}{a_{i+1}-a_i})d({a_{i+1}-a_i})\\
&-A_\tau(\doubledecker{u_1,\ldots,u_{i-1}+u_i,u_{i+1},\ldots,u_r}{a_1,\ldots,\;\;\;\;a_{i-1}\;\;\;,a_{i+1},\ldots,a_r};a_{r+1})F_\tau(\doubledecker{\;\;\;u_i}{a_{i-1}-a_i})d({a_{i-1}-a_i})\Big),
\end{align*}
where, $a_0=0, a_{r+1}=z$. 
\end{proposition}
\begin{proof}
\begin{align*}
&\frac{\partial}{\partial a_i}A_\tau\left(\doubledecker{u_1,\ldots,u_r}{a_1,\ldots,a_r};z\right)\\
&=\int_{0<z_1<\cdots<z_{i-1}<z_{i+1}<\cdots<z_r<z}F_\tau(\doubledecker{\;\;\;u_1}{z_1-a_1})dz_1\cdots \left(\frac{\partial}{\partial a_i}\int_{z_{i-1}}^{z_{i+1}}F_\tau(\doubledecker{\;\;\;u_i}{z_i-a_i})dz_i\right)\cdots F_\tau(\doubledecker{\;\;\;u_r}{z_r-a_r})dz_r\\
&=\int_{0<z_1<\cdots<z_{i-1}<z_{i+1}<\cdots<z_r<z}F_\tau(\doubledecker{\;\;\;u_1}{z_1-a_1})dz_1\cdots F_\tau(\doubledecker{\;\;\;u_{i-1}}{z_{i-1}-a_{i-1}})F_\tau(\doubledecker{\;\;\;u_i}{z_{i-1}-a_i})dz_{i-1}\cdots F_\tau(\doubledecker{\;\;\;u_r}{z_r-a_r})dz_r 
\\&-\int_{0<z_1<\cdots<z_{i-1}<z_{i+1}<\cdots<z_r<z}F_\tau(\doubledecker{\;\;\;u_1}{z_1-a_1})dz_1\cdots F_\tau(\doubledecker{\;\;\;u_i}{z_{i+1}-a_i})F_\tau(\doubledecker{\;\;\;u_{i+1}}{z_{i+1}-a_{i+1}})dz_{i+1}\cdots F_\tau(\doubledecker{\;\;\;u_r}{z_r-a_r})dz_r .
\end{align*}
Here, by using the Fay identity for Eisenstein-Kronecker series,
\begin{align*}
F_\tau(\doubledecker{\;\;\;u_{i-1}}{z_{i-1}-a_{i-1}})F_\tau(\doubledecker{\;\;\;u_i}{z_{i-1}-a_i})&=F_\tau(\doubledecker{u_{i-1}+u_i}{z_{i-1}-a_{i-1}})F_\tau(\doubledecker{\;\;\;u_i}{a_{i-1}-a_i})+F_\tau(\doubledecker{u_{i-1}+u_i}{z_{i-1}-a_{i}})F_\tau(\doubledecker{\;u_{i-1}}{a_{i}-a_{i-1}}),\\
F_\tau(\doubledecker{\;\;\;u_i}{z_{i+1}-a_i})F_\tau(\doubledecker{\;\;\;u_{i+1}}{z_{i+1}-a_{i+1}})&=F_\tau(\doubledecker{u_{i}+u_{i+1}}{z_{i+1}-a_{i}})F_\tau(\doubledecker{\;u_{i+1}}{a_{i}-a_{i+1}})+F_\tau(\doubledecker{u_{i}+u_{i+1}}{z_{i+1}-a_{i+1}})F_\tau(\doubledecker{\;\;u_{i}}{a_{i+}-a_{i}}).
\end{align*}
Substituting this into the iterated integrals and then summing up the partial derivatives with respect to each $a_i$, we get the desired result.
\end{proof}
By using Proposition \ref{prop:diff}, we obtain the following relation among elliptic multiple polylogarithms.
\begin{proposition}\label{prop:rel}
\begin{align*}
&A_\tau \left(\doubledecker{u_1,\ldots,u_{r-1},u_r}{\;0\;,\ldots,\;\;0\;\;\;,\;z};z\right)\\
&=\sum_{i=2}^r(-1)^i\zeta(i)A_\tau \left(\doubledecker{u_i,\ldots,u_{r-1}}{0\;,\ldots,\;0\;};z\right)-\sum_{i=0}^{r-1}A_\tau \left(\doubledecker{u_1,\ldots,u_{i-1},u_{i}+\cdots+u_r,-u_{i+1}-\cdots -u_r,u_{i+1},\ldots,u_{r-1}}{0\;,\hspace{2.6cm}\ldots\hspace{2.6cm},\;0\;};z\right).
\end{align*}
\end{proposition}
\begin{proof}
We use induction on $r$. When $r = 1$, we have:
\begin{align*}
A_\tau \left(\doubledecker{u_1}{z};z\right)&=\int_{0<t<z}F_\tau(\doubledecker{\;u_1}{t-z})dt\\
&=-\int_{0<t<z}F_\tau(\doubledecker{-u_1}{z-t})dt\\
&=-\int_{0<s<z}F_\tau(\doubledecker{-u_1}{\;s})ds\\
&=A_\tau \left(\doubledecker{-u_1}{\;0};z\right).
\end{align*}
Now, for $r > 1$, we have:
\begin{align*}
&\frac{d}{dz}A_\tau \left(\doubledecker{u_1,\ldots,u_{r-1},u_r}{\;0\;,\ldots,\;\;0\;\;\;,\;z};z\right)\\
&=A_\tau \left(\doubledecker{u_1,\ldots,u_{r-1}+u_r}{\;0\;,\ldots\;,\;\;\;\;\;\;z};z\right)F_\tau(\doubledecker{u_{r-1}}{\;\;z})+A_\tau \left(\doubledecker{u_1,\ldots,u_{r-1}+u_r}{\;0\;,\ldots\;,\;\;\;\;\;\;0};z\right)F_\tau(\doubledecker{u_{r}}{-z})\\
&=A_\tau \left(\doubledecker{u_1,\ldots,u_{r-1}+u_r}{\;0\;,\ldots\;,\;\;\;\;\;\;z};z\right)F_\tau(\doubledecker{u_{r-1}}{\;\;z})-A_\tau \left(\doubledecker{u_1,\ldots,u_{r-1}+u_r}{\;0\;,\ldots\;,\;\;\;\;\;\;0};z\right)F_\tau(\doubledecker{-u_{r}}{\;\;z})
\intertext{By using the induction hypothesis, }
&=\sum_{i=2}^{r-1}(-1)^i\zeta(i)A_\tau \left(\doubledecker{u_i,\ldots,u_{r-2}}{0\;,\ldots,\;0\;};z\right)F_\tau(\doubledecker{u_{r-1}}{\;\;z})\\
&-\sum_{i=0}^{r-2}A_\tau \left(\doubledecker{u_1,\ldots,u_{i-1},u_{i}+\cdots+u_r,-u_{i+1}-\cdots-u_r,u_{i+1},\ldots,u_{r-2}}{0\;,\hspace{2.6cm}\ldots\hspace{2.6cm},\;0\;};z\right)F_\tau(\doubledecker{u_{r-1}}{\;\;z})\\
&-A_\tau \left(\doubledecker{u_1,\ldots,u_{r-1}+u_r}{\;0\;,\ldots\;,\;\;\;\;\;\;0};z\right)F_\tau(\doubledecker{-u_{r}}{\;\;z}).
\end{align*}
Integrating both sides from $0$ to $z$ gives:
\begin{align*}
A_\tau \left(\doubledecker{u_1,\ldots,u_{r-1},u_r}{\;0\;,\ldots,\;\;0\;\;\;,\;z};z\right)
&=\mathrm{Reg}_{t\to0}^{-2\pi i}A_\tau \left(\doubledecker{u_1,\ldots,u_{r-1},u_r}{\;0\;,\ldots,\;\;0\;\;\;,\;t};t\right)+\sum_{i=2}^{r-1}(-1)^i\zeta(i)A_\tau \left(\doubledecker{u_i,\ldots,u_{r-1}}{0\;,\ldots,\;0\;};z\right)\\
&-\sum_{i=0}^{r-1}A_\tau \left(\doubledecker{u_1,\ldots,u_{i-1},u_{i}+\cdots+u_r,-u_{i+1}-\cdots-u_r,u_{i+1},\ldots,u_{r-1}}{0\;,\hspace{2.6cm}\ldots\hspace{2.6cm},\;0\;};z\right).
\end{align*}
Finally, by using the fact that $F_\tau(\doubledecker{u}{z})=1/z+O(1)$, we have:
\begin{align*}
&\mathrm{Reg}_{t\to0}^{-2\pi i}A_\tau \left(\doubledecker{u_1,\ldots,u_{r-1},u_r}{\;0\;,\ldots,\;\;0\;\;\;,\;t};t\right)\\
&=\mathrm{Reg}_{t\to0}^{-2\pi i}\mathrm{Reg}_{\epsilon\to0}^{-2\pi i}\int_{\epsilon}^{t-\epsilon}\frac{dz_1}{z_1}\cdot\frac{dz_{r-1}}{z_{r-1}}\frac{dz_{r}}{z_{r}-t}\\
&=\mathrm{Reg}_{t\to0}^{-2\pi i}\mathrm{Reg}_{\epsilon\to0}^{-2\pi i}\int_{\epsilon/t}^{1-\epsilon/t}\frac{dz_1}{z_1}\cdot\frac{dz_{r-1}}{z_{r-1}}\frac{dz_{r}}{z_{r}-1}=(-1)^r\zeta(r).
\end{align*}
Combining these results, we obtain the desired conclusion.
\end{proof}
For $\bm{l}=(l_1,\ldots,l_r)\in\N_0^r$, we define the rational function $P_{\bm{l}}(u_1,\ldots,u_r)\in\Q(u_1,\ldots,u_r)$ as follows:
\[
P_{\bm{l}}(u_1,\ldots,u_r):=\sum_{i=0}^{r-1}u_1^{l_1-1}\cdots u_{i-1}^{l_{i-1}-1}(u_{i}+\cdots+u_r)^{l_i-1}(-u_{i+1}-\cdots -u_r)^{l_{i+1}-1}u_{i+1}^{l_{i+2}-1}\cdots u_{r-1}^{l_r-1}.
\]
\begin{lemma}
For any $\bm{l}=(l_1,\ldots,l_r)\in\N_0^r$, $u_1\cdots u_rP_{\bm{l}}(u_1,\ldots,u_r)$ is a $\Z$-polynomial (i.e. $u_1\cdots u_rP_{\bm{l}}(u_1,\ldots,u_r)\in\Z[u_1,\ldots,u_r]$), and furthermore, it is a homogeneous polynomial of degree $l_1+\cdots+l_r$.
\end{lemma}
\begin{proof}
Homogeneity is obvious, so we proceed to prove $u_1\cdots u_rP_{\bm{l}}(u_1,\ldots,u_r)\in\Z[u_1,\ldots,u_r]$ by induction on $r$. For $r=1$, we have $uP_l(u)=u\left((-u)^{l-1}+u^{l-1}\right)=\left((-1)^{l-1}+1\right)u^l\in\Z[u]$. Now, we assume $r>1$. Then, 
\begin{align*}
&u_1\cdots u_rP_{\bm{l}}(u_1,\ldots,u_r)\\
&=u_1\cdots u_r\bigg(u_1^{l_1-1}P_{l_2,\ldots,l_r}(u_2,\ldots,u_r)+\left((u_1+\cdots+u_r)^{l_1-1}-u_1^{l_1-1}\right)(-u_2-\cdots-u_r)^{l_2-1}\\
&\cdot u_2^{l_3-1}\cdots u_{r-1}^{l_r-1}+(-u_1-\cdots-u_r)^{l_1-1}u_1^{l_2-1}\cdots u_{r-1}^{l_r-1}\bigg)\\
&=u_1^{l_1}u_2\cdots u_rP_{l_2,\ldots,l_r}(u_2,\ldots,u_r)+u_2^{l_3}\cdots u_{r-1}^{l_r}u_1u_r\Big(\left((u_1+\cdots+u_r)^{l_1-1}-u_1^{l_1-1}\right)\\&\cdot(-u_2-\cdots-u_r)^{l_2-1}+(-u_1-\cdots-u_r)^{l_1-1}u_1^{l_2-1}\Big).
\end{align*}
 By the induction hypothesis, $u_2\cdots u_rP_{l_2,\ldots,l_r}(u_2,\ldots,u_r)\in\Z[u_2,\ldots,u_r]$, and we only need to show that $u_1u_r\Big(\Big((u_1+\cdots+u_r)^{l_1-1}\\-u_1^{l_1-1}\Big)(-u_2-\cdots-u_r)^{l_2-1}+(-u_1-\cdots-u_r)^{l_1-1}u_1^{l_2-1}\Big)\in\Z[u_1,\ldots,u_r]$. The cases where $l_1,l_2\geq1$ are obvious. 
\begin{enumerate}
\renewcommand{\labelenumi}{(\roman{enumi})}
\item When $l_1\geq1,l_2=0$, 
\begin{align*}
&u_1u_r\left(\left((u_1+\cdots+u_r)^{l_1-1}-u_1^{l_1-1}\right)(-u_2-\cdots-u_r)^{l_2-1}+(-u_1-\cdots-u_r)^{l_1-1}u_1^{l_2-1}\right)\\
&=-u_1u_r\frac{(u_1+\cdots+u_r)^{l_1-1}-u_1^{l_1-1}}{u_2+\cdots+u_r}+u_r(-u_1-\cdots-u_r)^{l_1-1}u_1^{l_2}\in\Z[u_1,\ldots,u_r].
\end{align*}
\item When $l_1=0,l_2\geq1$, 
\begin{align*}
&u_1u_r\left(\left((u_1+\cdots+u_r)^{l_1-1}-u_1^{l_1-1}\right)(-u_2-\cdots-u_r)^{l_2-1}+(-u_1-\cdots-u_r)^{l_1-1}u_1^{l_2-1}\right)\\
&=u_1u_r\frac{(-u_2-\cdots-u_r)^{l_2-1}-u_1^{l_2-1}}{u_1+\cdots+u_r}-u_r(-u_2-\cdots-u_r)^{l_2-1}\in\Z[u_1,\ldots,u_r].
\end{align*}
\item When $l_1=l_2=0$, 
\begin{align*}
&u_1u_r\left(\left((u_1+\cdots+u_r)^{l_1-1}-u_1^{l_1-1}\right)(-u_2-\cdots-u_r)^{l_2-1}+(-u_1-\cdots-u_r)^{l_1-1}u_1^{l_2-1}\right)\\
&=-\frac{u_1u_r}{(u_1+\cdots+u_r)(u_2+\cdots+u_r)}+\frac{u_r}{u_2+\cdots+u_r}-\frac{u_r}{u_1+\cdots+u_r}\\
&=0.
\end{align*}
\end{enumerate}
This completes the induction, and thus, the claim holds. 
\end{proof}
For $\bm{k}=(k_1,\ldots,k_r),\bm{l}=(l_1,\ldots,l_r)\in\N_0^r$, let $c\langle\bm{l}\mid\bm{k}\rangle$ denote the coefficient of $u_1^{k_1}\cdots u_r^{k_r}$ in $u_1\cdots u_rP_{\bm{l}}(u_1,\ldots,u_r)$. In other words, 
\begin{align}
u_1\cdots u_rP_{\bm{l}}(u_1,\ldots,u_r)=\sum_{\bm{k}\in\N_0^r}c\langle\bm{l}\mid\bm{k}\rangle u_1^{k_1}\cdots u_r^{k_r}\label{eq:c}. 
\end{align} 
\begin{corollary}\label{prop:empl}
For any $\bm{k}=(k_1,\ldots,k_r)\in\N_0^r$, we have 
\begin{align*}
\Gamma_\tau\left(\doubledecker{k_1,\ldots,k_{r-1},k_r}{\;0\;,\ldots,\;\;0\;\;\;,\;z};z\right)
&=\sum_{i=2}^r(-1)^i\delta_{1,k_1}\cdots\delta_{1,k_{i-1}}\delta_{1,k_r}\zeta(i)\Gamma_\tau \left(\doubledecker{k_i,\ldots,k_{r-1}}{0\;,\ldots,\;0\;};z\right)\\
&-\sum_{\bm{l}\in\N_0^r}c\langle\bm{l}\mid\bm{k}\rangle\Gamma_\tau \left(\doubledecker{l_1,\ldots,l_{r}}{0\;,\ldots,\;0\;};z\right).
\end{align*}
\end{corollary}
\begin{proof}
It follows by comparing the coefficients of $u_1^{k_1-1}\cdots u_r^{k_r-1}$ on both sides of Proposition \ref{prop:rel}. 
\end{proof}
\section{Elliptic multiple zeta values}
In this section, we discuss Fay relations of elliptic multiple zeta values and regularization. In §3.1, we explain how taking the limit of the differential relations among elliptic multiple polylogarithms leads to Fay relations among elliptic multiple zeta values. In §3.2, By using the Fay relations among elliptic multiple zeta values, we provide a proof of the main theorem regarding regularization of elliptic multiple zeta values.
\subsection{Fay relations among elliptic multiple zeta values}
The relation among elliptic multiple polylogarithms obtained in Corollary \ref{prop:empl} yields the following relation among  elliptic multiple zeta values by taking the regularized limit $\mathrm{Reg}_{1-z\to0}^{-2\pi i}$. 
\begin{theorem}\label{emzv-fay}
For any $\bm{k}=(k_1,\ldots,k_r)\in\N_0^r$ where $r=1$ or $r>1, k_r\neq1$, we have 
\begin{align*}
I^A(\bm{k})
&=\sum_{i=2}^r(-1)^i\delta_{1,k_1}\cdots\delta_{1,k_{i-1}}\delta_{1,k_r}\zeta(i)I^A(k_i,\ldots,k_{r-1})\\
&-\sum_{\bm{l}\in\N_0^r}c\langle\bm{l}\mid\bm{k}\rangle I^A (l_1,\ldots,l_r).
\end{align*}
\end{theorem}
\begin{proof}
By using Lemma \ref{lem:zreg} (i) for the left hand side of Corollary \ref{prop:empl}, we obtain
\[
\mathrm{Reg}_{1-z\to0}^{-2\pi i}\Gamma_\tau\left(\doubledecker{k_1,\ldots,k_{r-1},k_r}{\;0\;,\ldots,\;\;0\;\;\;,\;z};z\right)=I^A(\bm{k}).
\]
By using Lemma \ref{lem:zreg} (ii) for the right hand side of Corollary \ref{prop:empl}, we obtain
\begin{align*}
&\mathrm{Reg}_{1-z\to0}^{-2\pi i}\left(\sum_{i=2}^r(-1)^i\delta_{1,k_1}\cdots\delta_{1,k_{i-1}}\delta_{1,k_r}\zeta(i)\Gamma_\tau \left(\doubledecker{k_i,\ldots,k_{r-1}}{0\;,\ldots,\;0\;};z\right)-\sum_{\bm{l}\in\N_0^r}c\langle\bm{l}\mid\bm{k}\rangle\Gamma_\tau \left(\doubledecker{l_1,\ldots,l_{r}}{0\;,\ldots,\;0\;};z\right)\right)\\
&=\sum_{i=2}^r(-1)^i\delta_{1,k_1}\cdots\delta_{1,k_{i-1}}\delta_{1,k_r}\zeta(i)I^A(k_i,\ldots,k_{r-1})-\sum_{\bm{l}\in\N_0^r}c\langle\bm{l}\mid\bm{k}\rangle I^A (l_1,\ldots,l_r).
\end{align*}
\end{proof}
\subsection{Regularization}
First, we investigate properties of $c\langle\bm{l}\mid\bm{k}\rangle$ in eq. \eqref{eq:c}. 
\begin{lemma}\label{lem:c}
{\rm(i)} $c\langle l_1,\ldots,l_r\mid k_1,\ldots,k_{r-2},0,k_r\rangle=\delta_{0,l_r}\,c\langle l_1,\ldots,l_{r-1}\mid k_1,\ldots,k_{r-2},k_r\rangle$.
{\rm(ii)} If $l_1$ is even, 
\begin{align*}
c\langle l_1,1,l_3,\ldots,l_r\mid k_1\ldots,k_r\rangle&=\delta_{l_1,k_1}(c\langle 1,l_3,\ldots,l_r\mid k_2,\ldots,k_r\rangle-\delta_{l_3,k_2}\cdots\delta_{l_r,k_{r-1}}\delta_{1,k_r}).
\end{align*}
\end{lemma}
\begin{proof}
\begin{enumerate}
\renewcommand{\labelenumi}{(\roman{enumi})}
\item
By definition, 
\begin{align*}
&P_{l_1,\ldots,l_r}(u_1,\ldots,u_r)\\&=u_{r-1}^{l_r-1}P_{l_1,\ldots,l_{r-1}}(u_1,\ldots,u_{r-2},u_{r-1}+u_r)+u_1^{l_1-1}\cdots u_{r-2}^{l_{r-2}-1}(u_{r-1}+u_r)^{l_{r-1}-1}(-u_r)^{l_r-1}.
\end{align*}
By multiplying both sides by $u_1\cdots u_r(u_{r-1}+u_r)$, we obtain
\begin{align*}
&u_1\cdots u_r(u_{r-1}+u_r)P_{l_1,\ldots,l_r}(u_1,\ldots,u_r)\\&=u_1\cdots u_{r-2}(u_{r-1}+u_r)u_{r-1}^{l_r}u_rP_{l_1,\ldots,l_{r-1}}(u_1,\ldots,u_{r-2},u_{r-1}+u_r)\\&-u_1^{l_1}\cdots u_{r-2}^{l_{r-2}}u_{r-1}(u_{r-1}+u_r)^{l_{r-1}}(-u_r)^{l_r}.
\end{align*}
By comparing coefficients of $u_1^{k_1}\cdots u_{r-2}^{k_{r-2}}u_{r}^{k_r+1}$ on both sides, we obtain the desired result. 
\item
By definition, 
\begin{align*}
P_{l_1,1,l_3\ldots,l_r}(u_1,\ldots,u_r)&=u_1^{l_1-1}P_{1,l_3\ldots,l_r}(u_2,\ldots,u_r)
\\&+\left((u_1+\cdots+u_r)^{l_1-1}-u_1^{l_1-1}\right)u_2^{l_3-1}\cdots u_{r-1}^{l_r-1}\\&+(-u_1-\cdots-u_r)^{l_1-1}u_2^{l_3-1}\cdots u_{r-1}^{l_r-1}
\intertext{Since $l_1$ is even, }
&=u_1^{l_1-1}\left(P_{1,l_3\ldots,l_r}(u_2,\ldots,u_r)-u_2^{l_3-1}\cdots u_{r-1}^{l_r-1}\right).
\end{align*}
By multiplying both sides by $u_1\cdots u_r$ and comparing coefficients of $u_1^{k_1}\cdots u_{r}^{k_r}$, we obtain the desired result. 
\end{enumerate}
\end{proof}
By using shuffle relation and reflection relation, we prove a lemma that will be used later. For $n\geq0$, we define the subspace $\mathcal{L}_{n}(\mathcal{E}^A\mathcal{Z})$ of $\mathcal{E}^A\mathcal{Z}$ as follows: 
\[
\mathcal{L}_{n}(\mathcal{E}^A\mathcal{Z}):=\langle I^A(k_1,\ldots,k_r)\mid r\leq n,k_1,\ldots,k_r\in\N_0\rangle_\Q.
\]
\begin{definition}
For $\bm{k}=(k_1,\ldots,k_r)\in\N_0^r\,(r>0)$, we define $\bm{k}$ to have ``even parity'' if $r+k_1+\cdots+k_r\equiv0\mod 2$ and ``odd parity'' otherwise. 
\end{definition}
\begin{lemma}[{cf. \cite{BMS}}]\label{lem:par}
When $\bm{k}$ has even parity, the following congruence holds:
\[
I^A(\bm{k})\equiv\frac{1}{2}\left(I^A(k_1)I^A(k_2,\ldots,k_r)+I^A(k_1,\ldots,k_{r-1})I^A(k_r)\right)\mod(\mathcal{L}_{r-2}(\mathcal{E}^A\mathcal{Z}))^2.
\]
\end{lemma}
\begin{proof}
By using the Hopf algebraic properties of $(\Q\langle\mathbf{e}\rangle,\shuffle,\Delta)$ with antipode $S$, for any $\bm{k}=(k_1,\ldots,k_r)\in\N_0^r\,(r>0)$, we have
\begin{align*}
0&=\sum_{i=0}^re_{k_1}\cdots e_{k_i}\shuffle S(e_{k_{i+1}}\cdots e_{k_r})=e_{k_1}\cdots e_{k_r}+S(e_{k_1}\cdots e_{k_r})\\&+\sum_{i=1}^{r-1}e_{k_1}\cdots e_{k_i}\shuffle S(e_{k_{i+1}}\cdots e_{k_r}). 
\end{align*}
Therefore, in the case where $\bm{k}$ has even parity, 
\begin{align*}
0&=2I^A(\bm{k})+\sum_{i=1}^{r-1}(-1)^{k_{i+1}+\cdots+k_r+r-i}I^A(k_1,\ldots,k_i)I^A(k_{i+1},\ldots,k_r)\\
&\equiv2I^A(\bm{k})+(-1)^{k_1-1}I^A(k_1)I^A(k_{2},\ldots,k_r)+(-1)^{k_{r}+1}I^A(k_1,\ldots,k_{r-1})I^A(k_r)\\&\mod(\mathcal{L}_{r-2}(\mathcal{E}^A\mathcal{Z}))^2.
\end{align*}
One checks that $I^A(k)=0$ for odd $k\in\N_0$. The result then follows. 
\end{proof}
\begin{lemma}\label{lem:trans}
For a index $\bm{k}=(k_1,\ldots,k_n)$ with $k_n\neq 1$ and $m>0$, we have
\[
I^A(\bm{k},\underbrace{1,\ldots,1}_{m})=\frac{(-1)^m}{m!}I^A(\underbrace{1\shuffle\cdots\shuffle1}_{m}\shuffle(k_1,\ldots,k_{n-1}),k_n). 
\]
\end{lemma}
\begin{proof}
By using the shuffle relation and $I^A(1)=0$, we have
\[
0=I^A(1)I^A(\bm{k}\underbrace{1,\ldots,1}_{m-1})=mI^A(\bm{k},\underbrace{1,\ldots,1}_{m})+I^A(1\shuffle(k_1,\ldots,k_{n-1}),k_n\underbrace{1,\ldots,1}_{m-1}). 
\]
Hence, 
\[
I^A(\bm{k},\underbrace{1,\ldots,1}_{m})=\frac{-1}{m}I^A(1\shuffle(k_1,\ldots,k_{n-1}),k_n\underbrace{1,\ldots,1}_{m-1}). 
\]
By iterating this transformation, we obtain the result. 
\end{proof}
The $\Q$-subspace $\mathcal{E}^A\mathcal{Z}_{\mathrm{adm}}\subset\mathcal{E}^A\mathcal{Z}$ is defined as follows: 
\[
\mathcal{E}^A\mathcal{Z}_{\mathrm{adm}}:=\langle I^A(k_1,\ldots,k_r)\mid r,k_1,\ldots,k_r\in\N_0,k_1,k_r\neq1\rangle_\Q.
\]
Additionally, for $w\geq0$, we define subspaces of weight $w$, denoted by $(\mathcal{E}^A\mathcal{Z})_w$ and $(\mathcal{E}^A\mathcal{Z}_{\mathrm{adm}})_w$, respectively:
\begin{align*}
(\mathcal{E}^A\mathcal{Z})_w&:=\langle I^A(k_1,\ldots,k_r)\mid r,k_1,\ldots,k_r\in\N_0,k_1+\cdots+k_r=w\rangle_\Q,\\
(\mathcal{E}^A\mathcal{Z}_{\mathrm{adm}})_w&:=\langle I^A(k_1,\ldots,k_r)\mid r,k_1,\ldots,k_r\in\N_0,k_1,k_r\neq1,k_1+\cdots+k_r=w\rangle_\Q. 
\end{align*}
One checks that $\mathcal{E}^A\mathcal{Z}_{\mathrm{adm}}$ is $\Q$-subalgebra of  $\mathcal{E}^A\mathcal{Z}$ under the shuffle product. 
\begin{lemma}\label{lem:fin}
If $\bm{k}=(1,k_2,\ldots,k_{r-1},0)$ has odd parity, then we have
\[
I^A(1,k_2,\ldots,k_{r-1},0)\equiv-I^A(0,1,k_2,\ldots,k_{r-1})\mod {\mathcal{E}^A\mathcal{Z}}_{\mathrm{adm}}+(\mathcal{L}_{r-2}(\mathcal{E}^A\mathcal{Z}))^2. 
\]
\end{lemma}
\begin{proof}
By Lemma \ref{lem:par} and $I^A(0)=1$, we have
\begin{align*}
I^A(0,1,k_2,\ldots,k_{r-1},0)&\equiv\frac{1}{2}\left(I^A(0)I^A(1,k_2,\ldots,k_{r-1},0)+I^A(0,1,k_2,\ldots,k_{r-1})I^A(0)\right)\\
&\equiv\frac{1}{2}\left(I^A(1,k_2,\ldots,k_{r-1},0)+I^A(0,1,k_2,\ldots,k_{r-1})\right)\\&\mod(\mathcal{L}_{r-2}(\mathcal{E}^A\mathcal{Z}))^2.
\end{align*}
Therefore, we get
\begin{align*}
I^A(1,k_2,\ldots,k_{r-1},0)&\equiv2I^A(0,1,k_2,\ldots,k_{r-1},0)-I^A(0,1,k_2,\ldots,k_{r-1})\\
&\equiv-I^A(0,1,k_2,\ldots,k_{r-1})\mod {\mathcal{E}^A\mathcal{Z}}_{\mathrm{adm}}+(\mathcal{L}_{r-2}(\mathcal{E}^A\mathcal{Z}))^2. 
\end{align*}
\end{proof}
\begin{theorem}\label{main}
As a $\Q$-algebra, the following equality holds:
\[
\mathcal{E}^A\mathcal{Z}=\mathcal{E}^A\mathcal{Z}_{\mathrm{adm}}[I^A(k_1,\ldots,k_r;\tau)\mid r\in\N_0,k_1,\ldots,k_r\in\{0,1\}]. 
\]
\end{theorem}
\begin{proof}
Let $\widetilde{\mathcal{E}^A\mathcal{Z}}_{\mathrm{adm}}=\mathcal{E}^A\mathcal{Z}_{\mathrm{adm}}[I^A(k_1,\ldots,k_r;\tau)\mid r\in\N_0,k_1,\ldots,k_r\in\{0,1\}]$. It is enough to show that $I^A(\bm{k})\in\widetilde{\mathcal{E}^A\mathcal{Z}}_{\mathrm{adm}}$ for any $r\geq1$ and index $\bm{k}=(k_1,\ldots,k_r)$. We will prove this by induction on the length $r$. For $r=1$, it is trivial since $I^A(1)=0$. Now, let us assume the statement holds for $r-1$ and consider $r>1$. Take an index $\bm{k}=(k_1,\ldots,k_r)\in\N_0^r$ such that $I^A(\bm{k})\notin\mathcal{E}^A\mathcal{Z}_{\mathrm{adm}}$. If $k_1\neq 1$ and $k_r=1$, we can use the reflection relation to set $k_1=1$. Further, Lemma \ref{lem:trans} allows us to assume $k_r\neq 1$. 
\par
When $\bm{k}$ has even parity, by Lemma \ref{lem:par}, we can express $I^A(\bm{k})$ as a polynomial of elliptic multiple zeta values with lengths less than or equal to $r-1$. By using the induction hypothesis, we conclude that $I^A(\bm{k})\in\widetilde{\mathcal{E}^A\mathcal{Z}}_{\mathrm{adm}}$. 
\par
Let $\bm{k}=(1,k_2,\ldots,k_r)\in\N_0^r$ be of odd parity, and assume $k_r\neq 1$. Define $\bm{k}'=(1,k_2,\ldots,k_{r-1},0,k_r)$. By Theorem \ref{emzv-fay}, we have
\begin{align*}
I^A(\bm{k}')
&=-\sum_{\bm{l}\in\N_0^{r+1}}c\langle\bm{l}\mid\bm{k}'\rangle I^A(\bm{l}).
\end{align*}
By Lemma \ref{lem:c} (i), we have $c\langle l_1,\ldots,l_r,l_{r+1}\mid\bm{k}'\rangle=\delta_{0,l_{r+1}}c\langle l_1,\ldots,l_r\mid\bm{k}\rangle$, and thus, the above expression becomes
\[
I^A(\bm{k}')
=-\sum_{\substack{\bm{s}\in\N_0^{r}}}c\langle \bm{s}\mid\bm{k}\rangle I^A(\bm{s},0). 
\]
By Lemma \ref{lem:par} and the induction hypothesis, we have
\begin{align*}
(\text{LHS}) &\equiv \frac{1}{2}I^A(k_r)I^A(1,k_2,\ldots,k_{r-1},0)\mod \widetilde{\mathcal{E}^A\mathcal{Z}}_{\mathrm{adm}}.\\
(\text{RHS}) &\equiv -\sum_{\substack{\bm{s}\in\N_0^{r}}}c\langle \bm{s}\mid\bm{k}\rangle I^A(\bm{s},0)\\
&\equiv -\sum_{\substack{\bm{s}\in\N_0^{r}}}\frac{1}{2}c\langle \bm{s}\mid\bm{k}\rangle\left(I^A(\bm{s})+I^A(s_1)I^A(s_2,\ldots,s_r,0)\right)\mod \widetilde{\mathcal{E}^A\mathcal{Z}}_{\mathrm{adm}}
\intertext{By using Theorem \ref{emzv-fay} for $-\sum_{\substack{\bm{s}\in\N_0^{r}}}\frac{1}{2}c\langle \bm{s}\mid\bm{k}\rangle I^A(\bm{s})$,}
&\equiv \frac{1}{2}I^A(\bm{k})-\sum_{\substack{\bm{s}\in\N_0^{r}}}\frac{1}{2}c\langle \bm{s}\mid\bm{k}\rangle I^A(s_1)I^A(s_2,\ldots,s_r,0)\mod \widetilde{\mathcal{E}^A\mathcal{Z}}_{\mathrm{adm}}
\intertext{For $\bm{s}=(s_1,\ldots,s_r)$, if $s_1$ is odd, then $I^A(s_1)=0$. If $s_1$ is even, By Lemma \ref{lem:c} (ii), $c\langle s_1,1,s_3,\ldots,s_r\mid 1,k_2,\ldots,k_r\rangle=0$. Therefore, we only need to consider the case where $s_1$ is even and $s_2\neq1$. In this case, only the terms $I^A(s_2,\ldots,s_r,0)\in\mathcal{E}^A\mathcal{Z}_{\mathrm{adm}}$ remain. Thus,}
&\equiv \frac{1}{2}I^A(\bm{k})\mod \widetilde{\mathcal{E}^A\mathcal{Z}}_{\mathrm{adm}}.
\end{align*}
Therefore,
\begin{align}\label{eq:key}
I^A(\bm{k}) &\equiv I^A(k_r)I^A(1,k_2,\ldots,k_{r-1},0)\mod \widetilde{\mathcal{E}^A\mathcal{Z}}_{\mathrm{adm}}. 
\end{align}
Since $I^A(\mbox{odd})=0$, we consider the case when $k_r$ is even. Then $(1,k_2,\ldots,k_{r-1},0)$ has odd parity. Let $k_n\notin\{0,1\}$ and $k_{n+1},\ldots,k_{r-1}\in\{0,1\}$. By Lemma \ref{lem:trans} and Lemma \ref{lem:fin},
we can express $I^A(1,k_2,\ldots,k_{r-1},0)$ as a linear combination of elliptic multiple zeta values whose index has $k_n$ at the rightmost position. Then, by applying the equation \eqref{eq:key} again, we can reduce it to the elliptic multiple zeta values with an index obtained by replacing $k_n\notin\{0,1\}$ with $0$. Repeating this process yields the desired result. 
\end{proof}
\centerline{Acknowledgement}
The author is deeply grateful to Professor Hidekazu Furusho for his careful reviewing of the paper and for his invaluable suggestions on its structure and details. He also sincerely thanks Professor Benjamin Enriquez for his helpful comments and encouragement. 
\bibliographystyle{amsalpha}
\bibliography{reference}
\end{document}